\documentclass{amsart}

\usepackage{amsfonts}
\usepackage{amsmath}
\usepackage{amssymb}
\usepackage{mathrsfs}
\usepackage{amsthm,dsfont}
\usepackage{graphicx}

\newtheorem{theorem}{Theorem}[section]

\newtheorem{lemma}[theorem]{Lemma}

\newtheorem{proposition}[theorem]{Proposition}

\theoremstyle{definition}
\newtheorem{definition}[theorem]{Definition}
\newtheorem{example}[theorem]{Example}

\theoremstyle{remark}
\newtheorem{remark}[theorem]{Remark}
\numberwithin{equation}{section}
\def\ac{{\textit {\textbf{Acknowledgement:}} }}

\begin{document}
\title{3-Dimensional Discrete curvature flows and discrete Einstein metric}
\author{Huabin Ge, Xu Xu, Shijin Zhang}
\address[Huabin Ge]{School of Mathematical Sciences and BICMR, Peking University, Beijing 100871, P.R. China}
\email{gehuabin@pku.edu.cn}

\address[Xu Xu]{School of Mathematics and Statistics, Wuhan University, Wuhan 430072, P.R. China}
\email{xuxu2@whu.edu.cn}

\address[Shijin Zhang]{School of Mathematics and Systems Science, Beihang University, Beijing 100871, P.R. China}
\email{zhangshj.1982@yahoo.com.cn}
\date{}

\begin{abstract}
We introduce the discrete Einstein metrics as critical points of discrete energy on triangulated 3-manifolds, and study them by discrete curvature flow of second (fourth) order. We also study the convergence of the discrete curvature flow. Discrete curvature flow of second order is an analogue of smooth Ricci flow.
\end{abstract}

\maketitle

\section*{introduction}
This paper is a continuation of paper [G3]. Given a 3-dimensional triangulated manifold with piecewise linear metric. In paper [CR], Cooper and Rivin endowed each vertex with the angle defect of solid angles, which is called combinatorial scalar curvature. In paper [Gn1], Glickenstein studied a type of discrete scalar curvature flow. Paper [G3] defined discrete quasi-Einstein metrics and gave some analytical conditions for the existence of discrete quasi-Einstein metrics by introducing two other types of discrete scalar curvature flows.

There are other methods to define discrete curvature, such as discrete edge curvature, which is defined at each edge, see [Gn2]. This curvature is somewhat like Ricci curvature. In this paper, we will define discrete Einstein metrics as the critical points of discrete total edge curvature functional. Moreover, we will introduce two types of discrete edge curvature flows, one is of second order, the other is of fourth order. Discrete edge curvature flow of second order may be considered as an analogue of smooth Ricci flow. However, discrete edge curvature flow of fourth order seems more powerful than the flow of second order. Our results rely deeply on the properties of discrete Laplacian.

\section{Discrete Ricci Curvature and Discrete Einstein metric}

\subsection{Space of PL-metrics}

Consider a compact manifold $M$ of dimension 3 with a triangulation $\mathcal{T}$ on $M$. The triangulation is written as $\mathcal{T}=\{V,E,F,T\}$, where $V,E,F,T$ represent the set of vertices, edges, faces and tetrahedrons respectively. Denote $v_{1},v_{2},\cdots,v_{N}$ as the vertices of $\mathcal{T}$, where $N$ is the number of the vertices. We often write $i$ instead of $v_{i}$. $l_{ij}$ is the length of the edge $\{i,j\}$ in $\mathcal{T}$ which is connecting the vertex $i$ and the vertex $j$.

PL-metric is a map $l:E\rightarrow (0,+\infty)$ such that for any $\{i,j,k,l\}\in T$, $l_{ij},l_{ik},l_{il},$ $l_{jk},l_{jl},l_{kl}$ can be realized as an Euclidean tetrahedron.

We may think of PL-metrics as points in $\mathds{R}^{m}_{>0}$, $m$ times Cartesian product of $(0,+\infty)$, where $m$ is the number of edges in $F$. However, not all points in $\mathds{R}^{m}_{>0}$ represent PL-metrics. So we still need some nondegenerate conditions. Consider a Euclidean tetrahedron $\{i,j,k,l\}\in T$ with edge lengths $l_{ij},l_{ik},l_{il},l_{jk},l_{jl},l_{kl}$, then the volume of the Euclidean tetrahedron $\{i,j,k,l\}$ has the following formula due to Tartaglia (1494)
\begin{equation}
\mathrm{V}_{ijkl}^{2}=\frac{1}{288}\det A_{ijkl},
\end{equation}
where $A_{ijkl}=
\begin{pmatrix}
0 & 1 & 1 & 1 & 1 \\
1 & 0 & l_{ij}^{2} & l_{ik}^{2} & l_{il}^{2}\\
1 & l_{ij}^{2} & 0 & l_{jk}^{2} & l_{jl}^{2}\\
1 & l_{ik}^{2} & l_{jk}^{2} & 0 & l_{kl}^{2}\\
1 & l_{il}^{2} & l_{jl}^{2} & l_{kl}^{2} & 0
\end{pmatrix}.
$

\begin{definition}
We say the Euclidean tetrahedron $\{i,j,k,l\}\in T$ is nondegenerate if it satisfies the following three conditions:
\begin{enumerate}
\item $l_{pq}>0$, for any $p,q \in \{i,j,k,l\}$;
\item $\mathrm{V}_{ijkl}>0$;
\item the triangle inequality holds for any triangle in the tetrahedron $\{i,j,k,l\}$, i.e., for any $p,q,r\in \{i,j,k,l\}$, $l_{pq}, l_{pr}, l_{qr}$ satisfies the triangle inequality.
    \end{enumerate}
The triangulation $\mathcal{T}$ is nondegenerate if for any Euclidean tetrahedron in $T$ is nondegenerate.
\end{definition}

Now we give the definition of the Euclidean PL-manifold.
\begin{definition}
The manifold $(M^{3}, \mathcal{T}, l)$ is called an Euclidean PL-manifold if there exists a nondegenerate Euclidean triangulation with map $l$.
\end{definition}
We denote
\begin{center}
$\mathfrak{M}_{l}\triangleq \{l:E\rightarrow (0, +\infty)$ such that $(M^{3}, \mathcal{T}, l)$ is an Euclidean PL-manifold\}.
\end{center}

$\mathfrak{M}_{\tau}\triangleq \{(l_{ij},l_{ik},l_{il},l_{jk},l_{jl},l_{kl})\in \mathds{R}^{6}_{>0}|$ the lengths of the edges in the tetrahedron $\tau=\{i,j,k,l\}\in T$ are $l_{pq}$\}.

Then $\mathfrak{M}_{\tau}\subseteqq \mathds{R}^{6}_{>0}$ is a simply connected open set. Set $\widetilde{\mathfrak{M}_{\tau}}\triangleq \mathfrak{M}_{\tau}\times \mathds{R}^{m-6}$, then $\mathfrak{M}_{l}=\cap_{\tau\in T}\widetilde{\mathfrak{M}_{\tau}}$, is also a open set.

We give an example about a triangulation $\mathcal{T}$ of the standard three dimensional sphere $\mathbb{S}^{3}$ embedding in $\mathds{R}^{4}$.
\begin{example}\label{ex}
We take $A_{1}=(1,0,0,0), A_{2}=(-1,0,0,0), B_{1}=(0,1,0,0), B_{2}=(0,-1,0,0), C_{1}=(0,0,1,0), C_{2}=(0,0,-1,0), D_{1}=(0,0,0,1), D_{2}=(0,0,0,-1)$ as the vertices of $\mathcal{T}$, $P_{i}Q_{j}(P\neq Q \in \{A, B,C,D\}, i,j=1,2)$ are the edges of $\mathcal {T}$, $P_{i}Q_{j}R_{k}(i,j,k=1,2, $ any two of $(P,Q,R)\in \{A,B,C,D\}$ are different) are the faces of $\mathcal{T}$, regular tetrahedrons $A_{i}B_{j}C_{k}D_{l}(i,j,k,l=1, 2)$ are the tetrahedrons of $\mathcal{T}$. We know all edges have same length, equal $\frac{\pi}{2}$. Hence $(\mathbb{S}^{3}, \mathcal{T}, l)(l=\frac{\pi}{2}\{1,\cdots,1\})$ is an Eucliden PL-manifold.
\end{example}

\subsection{Combinatorial Ricci curvature and total curvature functional}

Given a Euclidean tetrahedron $\{i,j,k,l\}\in T$, the dihedral angle at edge $\{i,j\}$ is denoted by $\beta_{ij,kl}$.
The combinatorial Ricci curvature of the cone metric at an edge is $2\pi$ minus the sum of dihedral angles at the edge, see [L]. Define $R_{ij}$ as the combinatorial Ricci curvature at the edge $\{i,j\}$, i.e.,
\begin{equation}
R_{ij}=2\pi-\sum_{\{i,j,k,l\}\in T}\beta_{ij,kl},
\end{equation}
where the sum is taken over all tetrahedrons having $\{i,j\}$ as one of its edges.

 For simplicity we will write $l_{ij}, R_{ij}$ as $l_{1}, \cdots, l_{m}, R_{1},\cdots, R_{m}$ respectively, they are supposed to be ordered one by one, $m$ is the number of the edges in $F$. Define $l=(l_{1},\cdots,l_{m})^{T}, R=(R_{1},\cdots,R_{m})^{T}$, the transpose of $(l_{1},\cdots,l_{m}), (R_{1},\cdots,R_{m})$ respectively.  We define the matrix $L$ as following
\begin{equation}
L=\frac{\partial (R_{1},\cdots,R_{m})}{\partial(l_{1},\cdots,l_{m})}
=
\begin{pmatrix}
\frac{\partial R_{1}}{\partial l_{1}} & \cdots  & \frac{\partial R_{1}}{\partial l_{m}}\\
\vdots   & \vdots  &   \vdots \\
\frac{\partial R_{m}}{\partial l_{1}} & \cdots & \frac{\partial R_{m}}{\partial l_{m}}
\end{pmatrix}
\end{equation}

The total curvature functional is defined by
 \begin{equation}
 S=\sum_{i=1}^{m}R_{i}l_{i}.
 \end{equation}
 And the discrete quadratic energy functional is defined by
 \begin{equation}
 \mathcal{C}(l)=\|R\|^{2}=\sum_{i=1}^{m}R_{i}^{2}.
 \end{equation}

$$dS=\sum_{i=1}^{m}R_{i}dl_{i}+\sum_{i=1}^{m}l_{i}dR_{i}.$$
By the Schl\"{a}fli formula
 $$\sum_{i=1}^{m}l_{i}dR_{i}=0,$$
we have
$$dS=\sum_{i=1}^{m}R_{i}dl_{i}.$$
So
 \begin{equation*}
 \nabla_{l}S=R,
 \end{equation*}
 \begin{equation*}
 \mathrm{Hess}_{l}S=L.
 \end{equation*}
From above we know the matrix $L$ is symmetric. It is easy to get
$$\frac{\partial \mathcal{C}}{\partial l_{j}}=2\sum_{i=1}^{m}\frac{\partial R_{i}}{\partial l_{j}}R_{i}=2(\frac{\partial R}{\partial l})^{T}R.$$
So
\begin{equation}\label{gradientC}
\nabla_{l}\mathcal{C}=2L^{T}R.
\end{equation}
Since $R(tl_{1},tl_{2},\cdots,tl_{m})=R(l_{1},l_{2},\cdots, l_{m})$, we obtain the Euler formula
\begin{equation}\label{EulerFormula}
Ll=0.
\end{equation}

The curvature $R_{ij}$ is a combinatorial analogue of Ricci curvature in smooth cases. Fix $i$,the sum of all $R_{ij}$ with $j$ connecting $i$ is the curvature $R_{i}$ defined by Cooper and Rivin, see [CR].

Next we will define the discrete Einstein metric.
\begin{definition}
A PL-metric $l$ is called a discrete Einstein metric, if there exists a number $\lambda$ such that $R=\lambda l$, it is denoted by $l_{DE}$ and $R_{DE}$.
\end{definition}
\begin{remark}
It is easy to compute the combinatorial Ricci curvature at any edge in $(\mathbb{S}^{3}, \mathcal{T}, l)$ as in the Example \ref{ex}, $R_{ij}=2\pi-4\arccos{1/3}$. So $R=\frac{4\pi-8\arccos(1/3)}{\pi}l$, $l$ is discrete Einstein metric.
\end{remark}
\subsection{Critical points of functionals}
We have defined the total curvature functional and discrete quadratical energy functional. Now we consider the average of the functionals $S_{r}=\frac{S}{\|l\|^{r}}, \mathcal{C}_{r}=\frac{\mathcal{C}}{\|l\|^{r}}$ for some $r$.

\begin{equation*}
\nabla_{l}S_{r}=\frac{1}{\|l\|^{r}}(R-\frac{rS}{\|l\|^{2}}l)=\frac{1}{\|l\|^{r}}(\nabla_{l}S-\frac{rS}{\|l\|^{2}}l).
\end{equation*}
%\begin{equation*}
%\nabla_{l}\mathcal{C}_{r}=\frac{1}{\|l\|^{r}}(\nabla_{l}\mathcal{C}-\frac{r\mathcal{C}}{\|l\|^{2}}l).
%\end{equation*}
Then $\nabla_{l}S_{r}=0$ if and only if $\nabla_{l}S-\frac{rS}{\|l\|^{2}}l=0$, i.e., $R=\frac{rS}{\|l\|^{2}}l$. Multiply both sides above equality by $l$, we have
$$S=rS.$$
Hence $r=1$ or $S=0$.

If $r=1$, then the critical point of $S_{1}$ is the discrete Einstein metric $l_{DE}$ such that $R_{DE}=\lambda_{DE}l_{DE}$. If $r\neq 1$, we have $S=0, R=0$, so the critical point is the discrete Ricci flat metric.

%And $\nabla_{l}\mathcal{C}_{r}=0$ if and only if $2L^{T}K-\frac{r\mathcal{C}}{\|l\|^{2}}l=0$. By using the Euler formula, $\nabla_{l}\mathcal{C}_{r}=0$ if and only if $r=0$ or $\mathcal{C}=0$. If we assume $\mathcal{C}\neq 0$, then the critical point of $C$ is the metric with $LK=0$.

\section{Combinatorial second order flow}

\subsection{Definition and evolution equations}
We define the combinatorial two-order flow
\begin{equation}\label{CRF}
\dot{l}(t)_{ij}=-R_{ij}, \quad\mathrm{or}\quad \dot{l}(t)=-R.
\end{equation}
We also consider the normalized combinatorial two-order flow
\begin{equation}\label{NCRF}
\dot{l}(t)_{ij}=-R_{ij}+\lambda l_{ij}, \quad\mathrm{or}\quad  \dot{l}(t)=-R+\lambda l
\end{equation}
where $\lambda=\frac{S}{\|l\|^{2}}$, $\|l\|^{2}=\sum_{i=1}^{n} l_{i}^{2}$. Since $\frac{d\|l\|^{2}}{dt}=0$, $\|l\|^{2}=constant$  is preserved by the normalized combinatorial two-order flow (\ref{NCRF}).

And it is easy to obtain the following evolution equations.
\begin{equation}
\dot{R}=\frac{\partial R}{\partial l}\dot{l}=L(-R+\lambda l)=-LR,
\end{equation}
where we have used the Euler formula $Ll=0$. So
\begin{equation}
\dot{\mathcal{C}}=-2R^{T}LR.
\end{equation}
\begin{equation*}
\begin{aligned}
\dot{S}&=\sum_{i=1}^{m}\dot{R_{i}}l_{i}+R_{i}\dot{l_{i}}=-\|R\|^{2}+\lambda S\\
&=\frac{S^{2}-\|l\|^{2}\|R\|^{2}}{\|l\|^{2}}=\frac{<R,l>^{2}-\|l\|^{2}\|R\|^{2}}{\|l\|^{2}}\\
&=-\|R-\lambda l\|^{2}=-\|R-\frac{S}{\|l\|^{2}}l\|^{2}\leq 0,
\end{aligned}
\end{equation*}
the above equality have used the Euler formula $Ll=0$. Hence
\begin{equation}
\dot{\lambda}=\frac{\dot{S}}{\|l\|^{2}}=-(\frac{\|R\|}{\|l\|})^{2}+\lambda^{2}=-\frac{\|R-\lambda l\|^{2}}{\|l\|^{2}}\leq 0.
\end{equation}

\subsection{Convergence of the combinatorial second order flow}
Finding good metrics is always a central topic in Riemannian geometry.
\begin{theorem}
If the solution of the flow (\ref{CRF}) exists for all time and converge to a non-degenerate PL-metric $l_{\infty}$, then the discrete Ricci flat metric exists. Moreover, $l_{\infty}$ is indeed one.
\end{theorem}
\begin{proof}
Under the combinatorial Ricci flow (\ref{CRF}), we have
\begin{equation*}
\dot{S}=-\|R\|^{2}\leq 0.
\end{equation*}
The existence of $\dot{S}_{\infty}$ follows from the existence of $l_{\infty}$. This implies $\dot{S}_{\infty}=0$ due to that $S$ is nonincreasing along the combinatorial Ricci flow (\ref{CRF}). Hence $R_{\infty}=0$.
\end{proof}

\begin{theorem}
If the solution of the flow (\ref{NCRF}) exists for all time and converge to a non-degenerate PL-metric $l_{\infty}$, then the discrete Einstein metric exists. Moreover, $l_{\infty}$ is indeed one.
\end{theorem}
\begin{proof}
Under the normalized combinatorial Ricci flow (\ref{NCRF}), we have
\begin{equation*}
\dot{\lambda}=-\frac{\|R-\lambda l\|^{2}}{\|l\|^{2}}\leq 0.
\end{equation*}
The existence of $\dot{\lambda}_{\infty}$ follows from the existence of $l_{\infty}$. This implies $\dot{\lambda}_{\infty}=0$ due to that $\lambda$ is nonincreasing along the normalized combinatorial Ricci flow (\ref{NCRF}). Hence $R_{\infty}-\lambda_{\infty}l_{\infty}=0$.
\end{proof}

\begin{theorem}
Given a nondegenerate metric $l$. Assume there exists a discrete Einstein metric $l_{DE}$ such that $R_{DE}=\lambda l_{DE}$ with $\lambda_{DE}(I_{m}-\frac{l_{DE}l_{DE}^{T}}{\|l_{DE}\|^{2}})-L_{DE}\leq 0$, where $I_{m}$ means the identity matrix of $m\times m$,  then there exists a small constant $\varepsilon >0$, if
\begin{equation*}
\|R(0)-R_{DE}\|<\varepsilon,
\end{equation*}
then the normalized combinatorial two-order flow (\ref{NCRF}) with initial metric $l(0)=l$ has long time existence and the metrics converge to the discrete Einstein metric $l_{DE}$.
\end{theorem}
\begin{proof}
We want to prove $l_{DE}$ is a local attractor of the flow. By the evolution equation of the combinatorial two-order flow
\begin{equation*}
\dot{l}=\Upsilon(l)=-R+\lambda l.
\end{equation*}
We have the ODE system
$$\Upsilon(l_{DE})=-R_{DE}+\lambda_{DE} l_{DE}=0,$$
and now we compute the differential of $\Upsilon (l)=-R+\lambda l$ at $l$.
\begin{equation}
\begin{aligned}
D_{l}\Upsilon(l)&=-D_{l}R+\lambda D_{l}l+lD_{l}\Upsilon\\
&=-L^{T}+\lambda I_{m}+l(\frac{L^{T}l+R}{\|l\|^{2}}-\frac{2Sl}{\|l\|^{4}})^{T}\\
&=\lambda I_{m}-L+\frac{lR^{T}}{\|l\|^{2}}-2S\frac{ll^{T}}{\|l\|^{4}}\\
&=\lambda (I_{m}-\frac{ll^{T}}{\|l\|^{2}})-L+\frac{l(R-\lambda l)^{T}}{\|l\|^{2}}
\end{aligned}
\end{equation}
we have used the symmetry of $L$ and the Euler formula in the third equality. So
$$D_{l}\Upsilon(l)|_{l=l_{DE}}=\lambda_{DE}(I-\frac{l_{DE}l_{DE}^{T}}{\|l_{DE}\|^{2}})-L_{DE}\leq 0.$$
So $R_{DE}$ is a local attractor of the flow. The system is asymptotically stable at $R_{DE}$. The following three conditions are equivalent
\begin{enumerate}
\item The initial metric $l(0)$ is close to $l_{DE}$;
\item The initial Ricci curvature $R(0)$ is close to $R_{DE}$;
\item The initial quadratic energy $\mathcal{C}(0)$ is close to the quadratic energy $\mathcal{C}(l_{DE})$ of the discrete Einstein metric $l_{DE}$.
\end{enumerate}
They both imply the long time existence and convergence of the (normalized) combinatorial two-order flow.

\end{proof}

\begin{remark}
Assuming $\lambda_1(L_{DE})>\lambda_{DE}$, the first eigenvalue of $L_{DE}$ at $l_{DE}$, satisfies $\lambda_1(L_{DE})>\lambda_{DE}$, then one get $\lambda_{DE}(I_{m}-\frac{l_{DE}l_{DE}^{T}}{\|l_{DE}\|^{2}})-L_{DE}\leq 0$, which can be shown by same methods form the proof of Theorem 5.4 in [G3].
\end{remark}

\section{Fourth order flow}
In this section, we consider the combinatorial four-order flow
\begin{equation}\label{LCF}
\dot{l}=-L^{T}R
\end{equation}
where $L^{T}$ means the transpose of $L$. Since (\ref{gradientC}), the combinatorial four-order flow can be showed as a gradient flow of energy $\mathcal{C}$, i.e.,
\begin{equation}
\dot{l}=-\frac{1}{2}\nabla_{l}\mathcal{C}=-L^{T}R.
\end{equation}
It is easy to obtain the flowing evolution equations:
\begin{equation}
\dot{R}=-LL^{T}R;
\end{equation}
\begin{equation}
\dot{\mathcal{C}}=-2R^{T}LL^{T}R=-2(L^{T}R)^{T}(L^{T}R)=-\frac{1}{2}||\nabla_{l}\mathcal{C}||^{2}\leq 0;
\end{equation}
\begin{equation}
\dot{S}=(\dot{R})^{T}l+R^{T}\dot{l}=-(LL^{T}R)^{T}l+R^{T}(-L^{T}R)=-R^{T}LL^{T}l-R^{T}L^{T}R=-R^{T}L^{T}R.
\end{equation}

And it is easy to compute the $L$ at the regular point $l=\{1,1,\cdot,1\}$ have the rank $m-1$. We consider the set of metric $l$ such that
$$\mathcal{M}_{l}^{*}\doteq \{l:E\rightarrow (0,+\infty)|(M^{3},\mathcal{T},l) \quad\mathrm{is}\quad PL-\mathrm{manifold\quad and}\quad \mathrm{rank}L=m-1\}.$$
Then $\mathcal{M}_{l}^{*}$ is an open set and the regular points in the set of $\mathcal{M}_{l}^{*}$.
\begin{theorem}
If the solution of the flow (\ref{LCF}) exists for all time and converge to a non-degenerate PL-metric $l_{\infty} \in \mathcal{M}_{l}^{*}$, then the discrete Einstein metric exists. Moreover, $l_{\infty}$ is indeed one.
\end{theorem}
\begin{proof}
Under the combinatorial four-order flow (\ref{LCF}), we have
\begin{equation*}
\dot{\mathcal{C}}\leq 0.
\end{equation*}
The existence of $\dot{\mathcal{C}}$ follows from the existence of $l_{\infty}$. This implies $\dot{\mathcal{C}}=0$ due to that $\mathcal{C}$ is nonincreasing along the combinatorial four-order flow (\ref{LCF}). Hence $L_{\infty}^{T}R_{\infty}=0$. Since $l_{\infty} \in \mathcal{M}_{l}^{*}$, $\mathrm{L}_{\infty}=m-1$ and $\ker{\mathrm{L}_{\infty}}=\{tl|t\in \mathds{R}\}$, so $R_{\infty}=\lambda_{\infty} l_{\infty}$ for some $\lambda_{\infty}$.
\end{proof}

\begin{theorem}
Assume there exists a discrete Einstein metric $l_{DE}\in \mathcal{M}_{l}^{*}$ such that $R_{DE}=\lambda l_{DE}$, then there exists a small constant $\varepsilon >0$, if
\begin{equation*}
\|R(0)-R_{DE}\|<\varepsilon,
\end{equation*}
then the combinatorial four-order flow
\begin{equation}\label{NLCF}
\dot{l}=L^{T}(R_{DE}-R)
\end{equation}
with initial data $l(0)$ has long time existence and the metrics converge to the discrete Einstein metric $l_{DE}$.
\end{theorem}
\begin{proof}
Under the normalized four-order flow (\ref{NLCF}),
$$\dot{R}=\frac{\partial R}{\partial l}\dot{l}=LL^{T}(R_{DE}-R),$$
$$\dot{\mathcal{C}}=-2(R_{DE}-R)^{T}LL^{T}(R_{DE}-R)\leq 0,$$
where $\mathcal{C}=\sum_{i=1}^{m}((R_{DE})_{i}-R_{i})^{2}$.
Now we consider the ODE system
$$\dot{l}=\Upsilon(l)=L^{T}(R_{DE}-R).$$
Then
$$\Upsilon(l_{DE})=0,$$
and
$$D_{l}\Upsilon(l)|_{l=l_{DE}}=-L_{DE}L_{DE}^{T}\leq 0.$$
Hence $R_{DE}$ is a local attractor, the system is asymptotically at $R_{DE}$.
\end{proof}

\section{$K$-space form triangulation and discrete curvature flow}

Assuming $K\in \mathds{R}$ is a constant, moreover, $K\neq 0$. In this section, we will consider a 3-dimensional compact manifold $M^3$ with a $K$-space form triangulation $\mathcal{T}$ on $M^3$. Let $M_K$ be the space form with constant sectional curvature $K$. The basic blocks of $K$-space form triangulation $\mathcal{T}$ are tetrahedrons embedded in $M_K$.

A tetrahedron embedded in $M_K$ is determined by its six edge lengths. Not every group of six positive numbers can be realized as the six edge lengths of some tetrahedrons embedded in $M_K$. Similar with Euclidean case, there are nondegenerate conditions too. However, we know that, all admissible group of six positive numbers which can be realized as the six edge lengths of some tetrahedrons embedded in $M_K$ make an open connected set in $\mathds{R}^6_{>0}$.

The combinatorial Ricci curvature $R_{ij}$ is defined as the same of the Euclidean PL-manifold. We need to define a new functional $S_K$ corresponding the total curvature functional $S$,
\begin{definition}
Define $V=\sum_{\{i,j,k,l\}\in T}\mathrm{V}_{ijkl}$ and define
$$S_K\triangleq 2KV+\sum_{i=1}^{m}R_{i}l_{i}.$$
\end{definition}

Now we recall the famous Schl\"{a}fli formula for a $K$-space form tetrahedra. For any $K$-space form tetrahedra $\{i,j,k,l\}\in T$, one have (see [Sc])
$$\frac{\partial \mathrm{V}_{ijkl}}{\partial \beta_{pq}}=\frac{l_{pq}}{2K}, p,q\in \{i,j,k,l\}.$$
Here $\beta_{pq}$ mean the dihedral angle at the edge $\{p,q\}$ in the tetrahedra $\{i,j,k,l\}$. Using above formula, one can get
$$2KdV+\sum_{i=1}^{m}l_{i}dR_{i}=0.$$
Hence
\begin{equation*}
dS_K=2KdV+\sum_{i=1}^{m}(l_{i}dR_{i}+R_{i}dl_{i})=\sum_{i=1}^{m}R_{i}dl_{i},
\end{equation*}
which implies $\frac{\partial S_K}{\partial l_i}=R_i$. Then we have

\begin{proposition}
$\quad$

\begin{enumerate}
\item $\nabla_{l}S_K=R$;
\item $\mathrm{Hess}_{l}S_K=L$.
\end{enumerate}
\end{proposition}

\begin{proposition}
$L$ is symmetric, nonsingular and indefinite.
\end{proposition}
\begin{proof}
L is the Jacobian of the functional $S_K$, and is always symmetric. We post the rest of the proof to the appendix, see Theorem \ref{L-nonsingular}.
\end{proof}

With $K$-space form triangulation, we consider discrete curvature flow $\dot{l}=-R$ of $2^{th}$ order and flow $\dot{l}=-L^{T}R$ of $4^{th}$ order. Most properties are laid out in the following table.\\

\begin{tabular}{c r l} \hline
Discrete curvature flow of $2^{th}$ order  & { Discrete curvature flow of $4^{th}$ order} \\\hline
$\dot{l}=-R=-\nabla_{l}S_K$ & $\dot{l}=-L^TR=\nabla_{l}\mathcal{C}$ \\
$\dot{R}=-L$ & $\dot{R}=-LL^TR$ \\
$\dot{S}_K=-R^TR=-\mathcal{C}\leq 0$ & $\dot{S}_K=-R^TL^TR$ \\
$\dot{\mathcal{C}}=-2R^TLR$ & $\dot{\mathcal{C}}=-2\|R^TL\|^2\leq 0$ \\\hline\\
\end{tabular}

\begin{theorem}
If the solution of discrete curvature flow $\dot{l}=-L^{T}R$ exist for all time and converge to a nondegenerate metric $l_{\infty}$, then the discrete Ricci curvature of the limit metric $l_{\infty}$ is zero. Moreover, $l_{\infty}$ is a discrete Ricci-flat hyperbolic metric.
\end{theorem}
\begin{proof}
$\lim_{t\rightarrow +\infty}\mathcal{C}(t)$ exists because of $l_{\infty}$, and $\mathcal{C}(t)$ is nonincreasing along the $4^{th}$ order discrete curvature flow, so we have
$$\lim_{t\rightarrow +\infty}\dot{\mathcal{C}(t)}=0.$$
So
$$\lim_{t\rightarrow +\infty}(L^{T}R)^{T}(L^{T}R)=0.$$
Hence $L^{T}R=0$. Since $L$ is nondegenerate, $R_{\infty}=0$.
\end{proof}
\begin{remark}
We can make same conclusion for discrete curvature flow $\dot{l}=-R$ of $2^{th}$ order.
\end{remark}

\begin{theorem}
If there exists a discrete Ricci-flat metric $l_{DE}$ with $R_{DE}=0$, then the solution of $4^{th}$ order discrete curvature flow $\dot{l}=L^{T}R$ exists for all time and converges to the discrete Einstein metric $l_{DE}$ when the initial discrete Calabi energy $\mathcal{C}(0)$ is small enough.
\end{theorem}
\begin{proof}
At point $l_{DE}$, $D_l(-L^TR)=-LL^T<0$. Hence $l_{DE}$ is a local attractor of the flow.
\end{proof}

\begin{remark}
It seams that $4^{th}$ order flow $\dot{l}=L^{T}R$ is more powerful than $2^{th}$ order flow $\dot{l}=-R$.
\end{remark}

\ac
The authors would like to thank Professor Feng Luo, Glickenstein for many helpful conversations. 

\section{Appendix}
In this section we will look at discrete Laplacians in space forms $M_K$, where subindex $K$ represents the constant sectional curvature. We conclude that the discrete Laplacian operator is nonsingular and indefinite whenever $K\neq 0$.

Consider a single tetrahedron $\tau=\{A,B,C,D\}$ embedded in $M_K$. $\tau$ varies with its six edge lengths, hence all tetrahedrons can be considered as points of some connected open set in $\mathds{R}^6_{>0}$. Denote $\beta_{AB}$ as the dihedral angle at edge $\{A,B\}$. The dihedral angles and the edge lengths are mutually determined. On one hand, six dihedral angles are determined by six edge lengths. On the other hand, each tetrahedron in the space form $M_K$ is determined, up to a motion, by its Gram matrix, which, in turn, is determined by the dihedral angles of the tetrahedron (see [Vi]). Therefore the Jacobian of dihedral angles over edges, which is denoted by
$$-L_{ABCD}\triangleq\frac{\partial (\beta_{AB},\beta_{AC},\beta_{AD},\beta_{BC},\beta_{BD},\beta_{CD})}{\partial (l_{AB},l_{AC},l_{AD},l_{BC},l_{BD},l_{CD})},$$
is nonsingular.

Next we prove that $L_{ABCD}$ is indefinite. A tetrahedron is called regular, if all lengths are equal.

\begin{proposition}
The discrete Laplacian matrix of a regular tetrahedron is
\begin{equation}
-L_{ABCD}=
\left(
 \begin{array} {cccccc}
   x & y & y & y & y & z\\
   y & x & y & y & z & y\\
   y & y & x & z & y & y\\
   y & y & z & x & y & y\\
   y & z & y & y & x & y\\
   z & y & y & y & y & x\\
  \end{array}
\right)
  \end{equation}
where $x=\frac{\partial \beta_{AB}}{\partial l_{AB}}$, $y=\frac{\partial \beta_{AB}}{\partial l_{AC}}=\frac{\partial \beta_{AB}}{\partial l_{AD}}=\frac{\partial \beta_{AB}}{\partial l_{BC}}=\frac{\partial \beta_{AB}}{\partial l_{BD}}$, $z=\frac{\partial \beta_{AB}}{\partial l_{CD}}$. Moreover, the eigenvalues of above $-L_{ABCD}$ are $x-z$, $x+z-2y$, $x+z+4y$ with degree 3, 2, 1 respectively.
\end{proposition}

\begin{figure}
\begin{center}
\includegraphics[width=4in,height=1.5in]{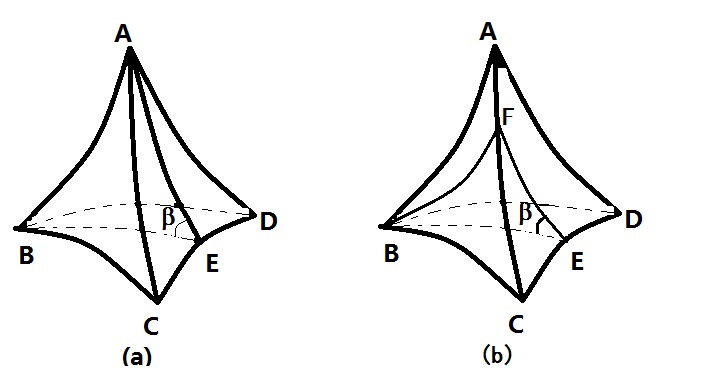}
\caption{}
\label{fig1}
\end{center}
\end{figure}

In the following, we claim that, when $K\neq 0$, the discrete Laplacians are nonsingular but not definite. It's enough to determine the signal of
$x-z$, $x+z-2y$ and $x+z+4y$.

First, we recall the formula of cosine law in the two-dimensional space forms $M^{2}(K)$ with constant sectional curvature $K$. Denote
\begin{eqnarray*}
S_{K}(t)=
\begin{cases}
\frac{\sin(\sqrt{K}t)}{\sqrt{K}},&K>0 \cr
t, &K=0 \cr
\frac{\sinh(\sqrt{-K}t)}{\sqrt{-K}},&K<0
\end{cases}
.
\end{eqnarray*}
\begin{eqnarray*}
C_{K}(t)=
\begin{cases}
\cos(\sqrt{K}t),&K>0 \cr
1, &K=0 \cr
\cosh(\sqrt{-K}t),&K<0
\end{cases}
.
\end{eqnarray*}
\begin{eqnarray*}
f_{K}(r)=\int_{0}^{r}S_{K}(t)dt=
\begin{cases}
\frac{1}{K}(1-C_{K}(r)), &K\neq 0 \cr
\frac{r^{2}}{2}, &K=0
\end{cases}.
\end{eqnarray*}
Then we have the following identity.
\begin{enumerate}
\item $f_{K}^{'}(r)=S_{K}(r), \quad S_{K}^{'}(r)=C_{K}(r);$
\item $KS_{K}^{2}(a)+C_{K}^{2}(a)=1;$
\item $S_{K}(a+b)=S_{K}(a)S_{K}(b)+C_{K}(a)C_{K}(b);$
\item $C_{K}(a+b)=C_{K}(a)C_{K}(b)-KS_{K}(a)S_{K}(b);$
\item $C_{K}(2a)=2C_{K}^{2}(a)-1=1-2KS_{K}^{2}(a).$
\end{enumerate}
So
\begin{equation}
f_{K}(r)=2S_{K}^{2}(r/2).
\end{equation}
\begin{proposition}[The Cosine law]
A geodesic triangle $\triangle ABC$ in the space form $M^{2}(K)$, with side lengths $a,b,c$,  is opposite angle $A,B,C$ respectively. Then the Cosine law is
\begin{equation}
f_{K}(c)=f_{K}(a-b)+S_{K}(a)S_{K}(b)(1-\cos C).
\end{equation}
For $K\neq 0$, the above formula is equivalent to
\begin{equation}
C_{K}(c)=C_{K}(a)C_{K}(b)+KS_{K}(a)S_{K}(b)\cos C.
\end{equation}
\end{proposition}

Now we calculate the exact value of $a, b, c$, we have the following results.
\begin{lemma}
\begin{equation}
z= \frac{\sqrt{2}C_{K}^{2}(l_{0}/2)}{S_{K}(l_{0}/2)\sqrt{1+3C_{K}(l_{0})}}.
\end{equation}
\end{lemma}
\begin{proof}
As the definition of $ L_{16}$, we just calculate $ \frac{\partial\beta}{\partial l_{6} }$. For calculate it, we only assume the length of $AB$ is $l_{6}$ and other edges have length $l_{0}$ in the hyperbolic tetrahedron in Figure 1(a). As shown in the Figure 1(a), $\mathrm{E}$ is the midpoint of the edge $\mathrm{CD}$, the dihedral angle at the edge $\mathrm{CD}$ is the angle $\angle\mathrm{AEB}$, i.e., $\beta$.

Using the cosine law in the triangle $\triangle\mathrm{AEB}$, we have
$$f_{K}(l_{6})=f_{K}(0)+S_{K}^{2}(h_{0})(1-\cos\beta_{1})=S_{K}^{2}(h_{0})(1-\cos\beta_{1}),$$
here $h_{0}$ is the length of the altitude in the regular triangle with side length $l_{0}$,
we can get
\begin{equation}
\frac{\partial\beta_{1}}{\partial l_{6} }= \frac{f_{K}^{'}(l_{6})}{S_{K}^{2}(h_{0})\sin \beta }=\frac{S_{K}(l_{6})}{S_{K}^{2}(h_{0})\sin \beta }.
\end{equation}
So at the regular point
$$z=\frac{S_{K}(l_{0})}{S_{K}^{2}(h_{0})\sin \beta }.$$
At the regular point,
$$f_{K}(l_{0})=f_{K}(h_{0}-\frac{l_{0}}{2})+S_{K}(h_{0})S_{K}(\frac{l_{0}}{2}),$$
then we have
$$C_{K}(h_{0})=\frac{C_{K}(l_{0})}{C_{K}(\frac{l_{0}}{2})}.$$
So we obtain
\begin{eqnarray*}
S_{K}^{2}(h_{0})=
\begin{cases}
\frac{1-C_{K}^{2}(h_{0})}{K}, &K\neq 0 \cr
h_{0}^{2}, & K=0
\end{cases}.
\end{eqnarray*}
For $K\neq 0$,
\begin{equation}
S_{K}^{2}(h_{0})=\frac{C_{K}^{2}(l_{0}/2)-C_{K}^{2}(l_{0})}{KC_{K}^{2}(l_{0}/2)}=\frac{S_{K}^{2}(l_{0}/2)(1+2C_{K}(l_{0}))}{C_{K}^{2}(l_{0}/2)},
\end{equation}
the above formula also holds for the case of $K=0$.
By the cosine law
\begin{equation}
\cos\beta=\frac{S_{K}^{2}(h_{0})-f_{K}(l_{0})}{S_{K}^{2}(h_{0})},
\end{equation}
If $K=0$, it is easy to get $\cos\beta=1-\frac{l_{0}^{2}}{2h_{0}^{2}}=1/3$. For the case of $K\neq 0$,
\begin{equation}\label{cosbeta}
\begin{aligned}
\cos\beta&=\frac{\frac{1-C_{K}^{2}(h_{0})}{K}-\frac{1-C_{K}(l_{0})}{K}}{\frac{1-C_{K}^{2}(h_{0})}{K}}
=\frac{C_{K}(l_{0})-C_{K}^{2}(h_{0})}{1-C_{K}^{2}(h_{0})}\\
&=\frac{C_{K}(l_{0})-\frac{C_{K}^{2}(l_{0})}{C_{K}^{2}(\frac{l_{0}}{2})}}{1-\frac{C_{K}^{2}(l_{0})}{C_{K}^{2}(\frac{l_{0}}{2})}}
=\frac{C_{K}(l_{0})(C_{K}^{2}(\frac{l_{0}}{2})-C_{K}(l_{0}))}{C_{K}^{2}(\frac{l_{0}}{2})-C_{K}^{2}(l_{0})}\\
&=\frac{KC_{K}(l_{0})S_{K}^{2}(\frac{l_{0}}{2})}{C_{K}^{2}(\frac{l_{0}}{2})-C_{K}^{2}(l_{0})}
=\frac{KC_{K}(l_{0})S_{K}^{2}(\frac{l_{0}}{2})}{\frac{1+C_{K}(l_{0}-2C_{K}^{2}(l_{0})}{2}}\\
&=\frac{KC_{K}(l_{0})S_{K}^{2}(\frac{l_{0}}{2})}{(1+2C_{K}(l_{0}))(\frac{1-C_{K}(l_{0})}{2})}
=\frac{KC_{K}(l_{0})S_{K}^{2}(\frac{l_{0}}{2})}{(1+2C_{K}(l_{0}))KS_{K}^{2}(l_{0}/2)}\\
&=\frac{C_{K}(l_{0})}{1+2C_{K}(l_{0})}.
\end{aligned}
\end{equation}
The above formula also holds for $K=0$.
Then we have
\begin{equation}\label{sinbeta}
\sin\beta=\frac{\sqrt{(1+C_{K}(l_{0}))(1+3C_{K}(l_{0}))}}{1+2C_{K}(l_{0})}=\frac{\sqrt{2}C_{K}(l_{0}/2)\sqrt{1+3C_{K}(l_{0})}}{1+2C_{K}(l_{0})}
\end{equation}
Hence
\begin{equation}
z=\frac{S_{K}(l_{0})C_{K}(l_{0}/2)}{\sqrt{2}S_{K}^{2}(l_{0}/2)\sqrt{1+3C_{K}(l_{0})}}=\frac{\sqrt{2}C_{K}^{2}(l_{0}/2)}{S_{K}(l_{0}/2)\sqrt{1+3C_{K}(l_{0})}}
\end{equation}
\end{proof}

\begin{lemma}

\begin{equation}
x=\frac{\sqrt{2}C_{K}^{2}(l_{0})}{S_{K}(l_{0}/2)\sqrt{1+3C_{K}(l_{0})}(1+2C_{K}(l_{0}))}.
\end{equation}
\end{lemma}

\begin{proof}
For calculate it, we only assume the length of $\mathrm{CD}$ is $l_{1}$ and other edges have length $l_{0}$ in the tetrahedron in Figure 1(a). As shown in the Figure 1(a), $\mathrm{E}$ is the midpoint of the edge $\mathrm{CD}$, the dihedral angle at the edge $\mathrm{CD}$ is the angle $\angle\mathrm{AEB}$, i.e., $\beta$. We assume the length of $\mathrm{AE}$ is $h$. By the cosine law,
$$f_{K}(l_{0})=f_{K}(0)+S_{K}^{2}(h)(1-\cos\beta)=S_{K}^{2}(h)(1-\cos\beta),$$
we get
\begin{equation}
-\frac{\partial \beta}{\partial l_{1}}=\frac{1-\cos\beta}{S_{K}^{2}(h)\sin\beta}\frac{\partial S_{K}^{2}(h)}{\partial l_{1}}.
\end{equation}
By the cosine law
$$f_{K}(l_{0})=f_{K}(l_{1}/2-h)+S_{K}(h)S_{K}(l_{1}/2),$$
we have
\begin{equation*}
C_{K}(h)=\frac{C_{K}(l_{0})}{C_{K}(l_{1}/2)}.
\end{equation*}
Hence
\begin{equation*}
\frac{\partial C_{K}(h)}{\partial l_{1}}=-\frac{C_{K}(l_{0})C'_{K}(l_{1}/2)}{2C_{K}^{2}(l_{1}/2)}=\frac{KS_{K}(l_{1}/2)C_{K}(l_{0})}{2C_{K}^{2}(l_{1}/2)}.
\end{equation*}
And
\begin{eqnarray}
S_{K}^{2}(h)=
\begin{cases}
\frac{1-C_{K}^{2}(h)}{K}, &K\neq 0\cr
h^{2}, &K=0
\end{cases},
\end{eqnarray}
we have
\begin{eqnarray*}
\frac{\partial S_{K}^{2}(h)}{\partial l_{1}}=
\begin{cases}
\frac{-2C_{K}(h)}{K}\frac{KS_{K}(l_{1}/2)C_{K}(l_{0})}{2C_{K}^{2}(l_{1}/2)}=-\frac{C_{K}^{2}(l_{0})S_{K}(l_{1}/2)}{C_{K}^{3}(l_{1}/2)}, &K\neq 0\cr
-l_{1}/2, &K=0
\end{cases}
\end{eqnarray*}
So we obtain
\begin{equation}
\frac{\partial S_{K}^{2}(h)}{\partial l_{1}}=-\frac{C_{K}^{2}(l_{0})S_{K}(l_{1}/2)}{C_{K}^{3}(l_{1}/2)}.
\end{equation}
At the regular point, we have
\begin{equation}
\begin{aligned}
x&=\frac{\partial\beta}{\partial l_{1}}=\frac{C_{K}^{2}(l_{0})S_{K}(l_{1}/2)}{C_{K}^{3}(l_{1}/2)}\frac{C_{K}(l_{0}/2)(1+C_{K}(l_{0}))}{\sqrt{2}S_{K}^{2}(l_{0}/2)\sqrt{1+3C_{K}(l_{0})}(1+2C_{K}(l_{0}))}\\
&=\frac{\sqrt{2}C_{K}^{2}(l_{0})}{S_{K}(l_{0}/2)\sqrt{1+3C_{K}(l_{0})}(1+2C_{K}(l_{0}))}.
\end{aligned}
\end{equation}

\end{proof}

\begin{lemma}
\begin{equation}
y= -\frac{\sqrt{2}C_{K}(l_{0})C_{K}^{2}(l_{0}/2)}{S_{K}(l_{0}/2)(1+2C_{K}(l_{0}))\sqrt{1+3C_{K}(l_{0})}}.
\end{equation}
\end{lemma}
\begin{proof}
For calculate it, we only assume the length of $\mathrm{AD}$ is $l_{2}$ and other edges have length $l_{0}$ in the tetrahedron in Figure 1(b). As shown in the Figure 1(b), $\mathrm{E}$ is the midpoint of the edge $\mathrm{CD}$, the dihedral angle at the edge $\mathrm{CD}$ is the angle $\angle\mathrm{FEB}$, i.e., $\beta$. For simplicity, we assume $l_{2}\leq 1_{0}$. Assume the length of $\mathrm{AF}$ is $s$, the length of $\mathrm{FE}$ is $\tilde{h}$. So the length of $\mathrm{FC}$ and $\mathrm{FD}$ are equal, is $l_{0}-s$. By the cosine law in the triangle $\triangle\mathrm{CEF}$,
\begin{equation}
f_{K}(l_{0}-s)=f_{K}(\tilde{h}-l_{0}/2)+S_{K}(\tilde{h})S_{K}(l_{0}/2),
\end{equation}
by the cosine law in the triangle $\triangle\mathrm{AFD}$,
\begin{equation}
f_{K}(l_{0}-s)=f_{K}(l_{2}-s)+\frac{S_{K}(s)}{S_{K}(l_{0})}(f_{K}(l_{0})-f_{K}(l_{2}-l_{0})),
\end{equation}
by the cosine law in the triangle $\triangle\mathrm{ABF}$ and $\triangle\mathrm{BEF}$,
\begin{equation}
f_{K}(l_{0}-s)+\frac{f_{K}(l_{0})S_{K}(s)}{S_{K}(l_{0})}=f_{K}(h_{0}-\tilde{h})+S_{K}(h_{0})S_{K}(\tilde{h})(1-\cos\beta).
\end{equation}

Differentiate the above three equations at the regular point,i.e., $ s=0, l_{2}=1_{0},\tilde{h}= h_{0} $. We get,
\begin{equation}
-S_{K}(l_{0})ds=(S_{K}(h_{0}-l_{0}/2)+C_{K}(h_{0})S_{K}(l_{0}/2))d\tilde{h}=S_{K}(h_{0})C_{K}(l_{0}/2)d\tilde{h},
\end{equation}
\begin{equation}
-S_{K}(l_{0})ds=-S_{K}(l_{0})ds+S_{K}(l_{0})dl_{2}+\frac{C_{K}(0)}{S_{K}(l_{0})}f_{K}(l_{0})ds,
\end{equation}
\begin{equation}
-S_{K}(l_{0})ds+\frac{f_{K}(l_{0})C_{K}(0)}{S_{K}(l_{0})}ds=-S_{K}(0)d\tilde{h}+S_{K}(h_{0})C_{K}(h_{0})(1-\cos\beta)d\tilde{h}+S_{K}^{2}(h_{0})\sin\beta d\beta.
\end{equation}
Using the fact $S_{K}(0)=0, C_{K}(0)=1$, we obtain
\begin{enumerate}
\item $ds=-\frac{S_{K}^{2}(l_{0})}{f_{K}(l_{0})}dl_{2},$
\item $d\tilde{h}=-\frac{S_{K}(l_{0})}{S_{K}(h_{0})C_{K}(l_{0}/2)}ds=\frac{S_{K}^{3}(l_{0})}{f_{K}(l_{0})S_{K}(h_{0})C_{K}(l_{0}/2)}dl_{2},$
\item $\frac{f_{K}(l_{0})-S_{K}^{2}(l_{0})}{S_{K}(l_{0})}ds=S_{K}(h_{0})C_{K}(h_{0})(1-\cos\beta)d\tilde{h}+S_{K}^{2}(h_{0})\sin\beta d\beta.$
\end{enumerate}
Using $\cos\beta=\frac{C_{K}(l_{0})}{1+2C_{K}(l_{0})}, C_{K}(h_{0})=\frac{C_{K}(l_{0})}{C_{K}(l_{0}/2)}$, we have
\begin{equation*}
-\frac{(f_{K}(l_{0})(1+2C_{K}(l_{0}))-S_{K}^{2}(l_{0}))S_{K}(l_{0})}{f_{K}(l_{0})(1+2C_{K}(l_{0}))}dl_{2}=S_{K}^{2}(h_{0})\sin\beta d\beta.
\end{equation*}
Since $f_{K}(l_{0})=2S_{K}^{2}(l_{0}/2)$, we have
$$-\frac{C_{K}(l_{0})S_{K}(l_{0})}{1+2C_{K}(l_{0})}dl_{2}=S_{K}^{2}\sin\beta d\beta.$$
Hence
\begin{equation}
\begin{aligned}
y=\frac{\partial \beta}{\partial l_{2}}&=-\frac{C_{K}(l_{0})S_{K}(l_{0})}{1+2C_{K}(l_{0})}\frac{1}{S_{K}^{2}(h_{0})\sin\beta}\\
&=-\frac{C_{K}(l_{0})S_{K}(l_{0})}{1+2C_{K}(l_{0})}\frac{C_{K}(l_{0}/2)}{\sqrt{2}S_{K}^{2}(l_{0}/2)\sqrt{1+3C_{K}(l_{0})}}\\
&=-\frac{\sqrt{2}C_{K}(l_{0})C_{K}^{2}(l_{0}/2)}{S_{K}(l_{0}/2)(1+2C_{K}(l_{0}))\sqrt{1+3C_{K}(l_{0})}}.
\end{aligned}
\end{equation}
\end{proof}

So we have
\begin{enumerate}
\item $x-y=-\frac{\sqrt{2}\sqrt{1+3C_{K}(l_{0})}}{2S_{K}(l_{0}/2)(1+2C_{K}(l_{0}))}<0,$
\item $x+z-2y=\frac{\sqrt{2}\sqrt{1+3C_{K}(l_{0})}}{2S_{K}(l_{0}/2)}>0,$
\item $x+z+4y=\frac{\sqrt{2}KS_{K}(l_{0}/2)}{(1+2C_{K}(l_{0}))\sqrt{1+3C_{K}(l_{0})}}.$
\end{enumerate}
Hence $x+z+4y>0$, when $K>0$; $x+z+4y=0$, when $K=0$; $x+z+4y<0$, when $K<0$.

\begin{theorem}
When $K\neq 0$, the discrete Laplacian of one single tetrahedron $-L_{ABCD}$ embedded in $M_K$ is nonsingular and indefinite.
\end{theorem}
\begin{proof}
Form above calculation we know, the discrete Laplacian matrix at regular point is indefinite. Any tetrahedron can be deformed continuously to the regular tetrahedron, so all tetrahedron must have the same properties.
\end{proof}

\begin{theorem}\label{L-nonsingular}
Consider a 3-dimensional compact manifold $M$ with a $K$-space form triangulation $\mathcal{T}$, where $K\neq 0$. The discrete Laplacian $L$ is nonsingular and indefinite.
\end{theorem}
\begin{proof}
By adding zeroes to the other entries, we can extend $6\times 6$ matrix $L_{ABCD}$ to a $N\times N$ matrix  which is still denoted as $L_{ABCD}$ without confusion. Then $L$ is the inner direct sum of all such $L_{ABCD}$, where $\{A,B,C,D\}$ is any tetrahedron in the triangulation $\mathcal{T}$.
\end{proof}

\bibliographystyle{amsplain}

\end{document}